\documentclass{article}
\usepackage{amsfonts}
\usepackage{amsmath}
\usepackage{amssymb}
\usepackage{fullpage}
\usepackage{amsthm}
\usepackage{graphicx}
\usepackage{relsize}
\usepackage{todonotes}
\usepackage{hyperref}
\usepackage[capitalise]{cleveref}
\usepackage{xcolor}

\newtheorem{theorem}{Theorem}[section]
\newtheorem{proposition}[theorem]{Proposition}
\newtheorem{lemma}[theorem]{Lemma} 
\newtheorem{corollary}[theorem]{Corollary}

\newtheorem{conj}[theorem]{Conjecture}

\theoremstyle{definition}
\newtheorem{definition}[theorem]{Definition}

\newcommand{\N}{{\mathbb N}}
\newcommand{\R}{{\mathbb R}}

\newcommand{\Z}{{\mathbb Z}}

\newcommand{\lk}{\operatorname{Lk}}

\title{Finitely generated groups acting uniformly properly on hyperbolic space}

\author{Robert Kropholler and Vladimir Vankov}

\begin{document}
\maketitle

\begin{abstract}
We construct an uncountable sequence of groups acting uniformly properly on hyperbolic spaces. We show that only countably many of these groups can be virtually torsion-free. This gives new examples of groups acting uniformly properly on hyperbolic spaces that are not virtually torsion-free and cannot be subgroups of hyperbolic groups.
\end{abstract}

\section{Introduction}

We say that a group $G$ acts {\em properly} on a metric space $X$ if for all $r>0$ and $x\in X$ there exists $N$ such that $|\{g\in G\mid d(x, gx)\leq r\}|\leq N.$
A group is said to be {\em hyperbolic} if it acts properly and cocompactly by isometries on a hyperbolic metric space. 
It is currently not known whether hyperbolic groups which are not virtually torsion-free exist.
There have been many interesting generalisations of this notion. 
For instance: acylindrically hyperbolic groups \cite{osin} or relatively hyperbolic groups \cite{gromov}. 
In this paper we will study the class of groups with uniformly proper actions studied in \cite{coulonosin}. 

\begin{definition}
	Let $G$ be a group acting on a metric space $X$. 
	We say that the action is {\em uniformly proper} if for every $r>0$ there exists $N$ such that for all $x\in X$, 
	$$|\{g\in G\mid d(x, gx)\leq r\}|\leq N.$$
\end{definition}
One should note that $N$ in this definition only depends on $r$ and not on $x$.
We insist on this, as otherwise all groups admit a proper action a hyperbolic space, namely their combinatorial horoball \cite{groves}.

The class  of groups acting uniformly properly on hyperbolic spaces includes all subgroups of hyperbolic groups. 
In \cite{coulonosin}, the question of whether these two classes coincide is asked. 
This is shown to be false in \cite{kropholler}, where uncountably many groups of type $FP_2$ acting uniformly properly on hyperbolic spaces are obtained, while there are only countably many finitely generated subgroups of hyperbolic groups. 

In this paper, we provide a streamlined version of this construction and go on to show that many of the examples created are not virtually torsion-free. 
Specifically, we prove the following:
\begin{theorem}\label{mainthm}
	There exist uncountably many finitely generated groups $H_W$ acting uniformly properly on hyperbolic spaces. 
	Moreover, at most countably many of the groups obtained are virtually torsion-free. 
\end{theorem}

The construction starts by making a subgroup $H$ of a hyperbolic group $G$ that is finitely generated but not finitely presented. 
The presentation for $H$ has relators contained in a set $V$. 
By considering a subset $W\subset V$, we obtain a group $H_W$ by replacing relators in $W$ with their $p$-th powers. 

The algebraic structure of the groups constructed is similar to that of subgroups of hyperbolic groups. 
Thus it would be of interest to obtain a characterisation of when these groups embed in hyperbolic groups. 

For example, if the set of relators for which $p$-th powers are taken is chosen in a periodic way, then the group $H_W$ is a subgroup of a hyperbolic group. 
In this instance, the hyperbolic group acts properly cocompactly on a {CAT(0)} cube complex and so is virtually special \cite{agol}.  Hence, $H_W$ is virtually torsion-free in this case. 

One should compare this to \cite{vankov}, where a similar construction is applied to generalised Bestvina-Brady groups. 
There it is shown that if $H_W$ is virtually torsion-free, then the set of relators for which $p$-th powers are taken has to be periodic. 
Therefore we may conjecture that this is the case here as well.
\begin{conj}\label{conj}
	The group $H_W$ is virtually torsion-free if and only if $W$ is a periodic subset of $V$. 

	This is the case if and only if $H_W$ embeds in a hyperbolic group. 
\end{conj}

A similar conjecture regarding generalised Bestvina-Brady groups is to appear in a forthcoming paper of Leary and the second author. However, it is interesting to note that it is the converse which is open in that setting.

In Section 2, we construct a non-positively curved cube complex $X$ whose fundamental group $G$ is hyperbolic. 
In Section 3, we give a Morse function on $X$ and find a subgroup of $G$ that is finitely generated but not finitely presented. 
In Section 4, we take branched covers of the cube complex $X$ to obtain uncountably many isomorphism classes of groups acting uniformly properly on hyperbolic spaces. 
Finally, in Section 5, we provide a criterion for when these groups are virtually torsion-free and show that this criterion can only be satisfied by countably many of the groups constructed in Section 4. 

We thank Ian Leary for excellent hospitality at the University of Southampton, where the ideas for this paper began. We also thank the organisers of YGGT VIII, where the authors first met.

\section{A construction of a hyperbolic group}

We begin by constructing one hyperbolic group with a subgroup that is finitely generated but not finitely presented. 
We will then take branched covers of a cube complex to obtain our sequence of groups with the desired properties. 

The following construction can be found in \cite{lodha}.
We use a sizeable graph constructed in the appendix of \cite{gardam} for a figure of this graph see Figure 6.1 on page 102 of \cite{gardamthesis}. 

Let $\Gamma$ be the graph with vertex set $A\sqcup B$, where $A = A^-\sqcup A^+$ and $B = B^-\sqcup B^+$, with $A^-, A^+, B^-, B^+ = \Z/9\Z$. 
There is an edge $a$ to $b$ if any of the following hold:
\begin{enumerate}
	\item If $a\in A^+, b\in B^+$, then $a = b$ or $a = b+1$.
	\item If $a\in A^+, b\in B^-$, then $a = b$ or $a = b-2$. 
	\item If $a\in A^-, b\in B^+$, then $a = b$ or $a = b+2$.
	\item If $a\in A^-, b\in B^-$, then $a = b+1$ or $a = b+2$. 
\end{enumerate}

\begin{proposition}
	$\Gamma$ has no embedded loops of length $<5$. 

	The full subgraph of $\Gamma$ spanned by $A^s\sqcup B^t$ is a loop of length 18, for all choices of $s, t$. 
\end{proposition}
\begin{proof}	
	These conditions correspond to modularity conditions $\mod 9$. 
	For full details see the appendix in \cite{gardam}
\end{proof}

Let $\Lambda_A$ be the graph with two vertices $a^+, a^-$ and $|A|$ edges each running from $a^-$ to $a^+$. 
Define $\Lambda_B$ similarly. 
The squares in $\Lambda_A\times \Lambda_B$ are in one-to-one correspondence with $A\times B$.

Let $X_{\Gamma}\subset \Lambda_A\times\Lambda_B$ be the cubical subcomplex containing $(\Lambda_A\times\Lambda_B)^{(1)}$ and those squares $(a, b)$ such that $(a, b)$ is an edge of $\Gamma$.

\begin{proposition}
	The link of any vertex of $X_{\Gamma}$ is $\Gamma$. 
\end{proposition}
\begin{proof}
	There are 4 vertices in $X_{\Gamma}$ but the definition permits a symmetry taking any vertex to any other. 
	Thus we will focus on the case of the vertex $v = (a^+, b^+)$. 
	Let $L = \lk(v, X_{\Gamma})$. 
	
	Since $(\Lambda_A\times \Lambda_B)^{(1)}\subset X_{\Gamma}$, we see that $V(L) = A\sqcup B$. 
	There is an edge in $L$ from the vertex $a$ to the vertex $b$ exactly when there is a square at $v$ with edges $a, b$. 
	We can see that this is exactly the case when $(a, b)$ is an edge of $\Gamma$. 
\end{proof}

Since $\Gamma$ is triangle-free, we deduce that $X_{\Gamma}$ is a non-positively curved cube complex. We can now apply a theorem of Moussong \cite{moussong}.

\begin{theorem}
	Let $X$ be a 2-dimensional non-positively curved cube complex. 
	Suppose that the link of each vertex does not contain an embedded loop of length 4. 
	Then $\pi_1(X)$ is hyperbolic. 
\end{theorem}

We already know that $\Gamma$ has no cycles of length less than 5, therefore we can conclude:

\begin{corollary}\label{cor:hyp}
	$\pi_1(X_{\Gamma})$ is hyperbolic. 
\end{corollary}

\section{The first example of a finitely generated not finitely presented subgroup}

To find the first example of a subgroup that is finitely generated but not finitely presented, we will use Morse theory. 
For details on Morse theory, see \cite{bbmorse}. 

Throughout, the circle $S^1$ will be triangulated with one vertex and one edge, which we will give an arbitrary orientation. 

Give each edge of $\Lambda_A$ an orientation by 
orienting towards $a^+$ if $a\in A^+$ and towards $a^-$ if $a\in A^-$. 
Define an orientation on $\Lambda_B$ similarly. 

Define $f_A\colon \Lambda_A\to S^1$ by mapping each edge linearly via its orientation. 
Define $f_B$ similarly. 
We can now define a map $f\colon \Lambda_A\times\Lambda_B\to S^1 = \R/\Z$ by $(x, y)\mapsto f_A(x)+f_B(y)$. 
We can restrict to $X_{\Gamma}$ in order to define a function which we also denote $f\colon X_{\Gamma}\to S^1$. 

Lifting this function to $\widetilde{X_{\Gamma}}$ gives us a Morse function $f\colon \widetilde{X_{\Gamma}}\to\R$.

Since the local geometry at each vertex of $\widetilde{X_{\Gamma}}$ is determined by that of $X_{\Gamma}$, we will abuse notation and label the vertices in $\widetilde{X_{\Gamma}}$ by the vertex they map to in $X_{\Gamma}$. 

\begin{proposition}\label{prop:ascdesclinks}
	The descending link of $(a^s, b^t)$ is the full subgraph of $\Gamma$ spanned by $A^s\sqcup B^t$.
	The ascending link of $(a^s, b^t)$ is the full subgraph of $\Gamma$ spanned by $A^{-s}\sqcup B^{-t}$. 
	Hence, both the ascending and descending links are cycles of length $18$. 
\end{proposition}
\begin{proof}
	The descending link of $v$ is the full subgraph of $\Gamma$ spanned by the vertices corresponding to edges oriented towards $v$. 
	In the case of $(a^s, b^t)$ these are exactly the edges coming from $A^s\sqcup B^t$. 
	Thus we obtain the desired result. 

	The ascending link is computed similarly, with the signs reversed. 
\end{proof}

\begin{theorem}
	Let $X = X_{\Gamma}$ and $f$ be the complex and Morse function from above. 
	Let $H$ be the kernel of $f_*$. 
	Then $\pi_1(X)$ is a hyperbolic group and $H$ is finitely generated and not finitely presented. 
\end{theorem}
\begin{proof}
	By \cref{cor:hyp}, we know that $\pi_1(X)$ is a hyperbolic group. 
	By \cref{prop:ascdesclinks}, we have that the ascending and descending links are copies of $S^1$. 
	Thus they are connected graphs and so $H$ is finitely generated by \cite[Theorem 4.1]{bbmorse}. 
	To see that $H$ is not finitely presented we can use \cite[Theorem 4.7]{brady}, which shows that it is not of type $FP_2$. 
\end{proof}

It will be useful in the next sections to consider a presentation for $H$. 
Let $Z$ be the cyclic cover of $X_{\Gamma}$ corresponding to $H$. 
There is a map $h\colon Z\to \R$ which is the lift of $f$ to $Z$. 
Let $Z_0$ denote the preimage of $0$. 
This is a graph and the inclusion $Z_0\to Z$ gives a surjection on fundamental groups. 
From \cite{bbmorse}, we know that as we expand $Z_0$ to $h^{-1}([-t, t])$, the homotopy changes by coning off the ascending and descending links of vertices.
Since each ascending or descending link is a copy of $S^1$, we can see that each vertex $v$ with $h(v)\neq 0$ gives one relation in the presentation of $H$. 

Let $v$ be a vertex of $Z$ such that $h(v)>0$. 
Thus we obtain a relation in $H$ from $v$ which corresponds to coning off the descending link of $v$. 
We can think about the descending link of $v$ as a subset of $Z$ in the $\frac{1}{4}$-neighbourhood of $v$. 
Denote this loop by $\gamma_v$. 
We can homotope $\gamma_v$ to $Z_0$ in order to obtain the relation $r_v$ corresponding to $v$. 

Later it will be useful to know that $\gamma_v$ is not trivial in $Z\smallsetminus\{v\}$. 
\begin{lemma}\label{lem:necrels}
	Let $\gamma_v$ be the loop described above. 
	Then $\gamma_v$ is not homotopically trivial in $Z\smallsetminus\{v\}$. 
\end{lemma}
\begin{proof}
	Let $\widetilde{Z}$ be the universal cover of $Z$ and $\rho$ be the associated covering map. 
	We obtain another covering $q\colon \widetilde{Z}\smallsetminus \rho^{-1}(v)\to Z\smallsetminus \{v\}$. 
	Note that $\widetilde{Z}$ is also the universal cover of $X_{\Gamma}$, so is a CAT(0) cube complex.
	Since $\gamma_v$ is a null-homotopic loop in $Z$, we can see that it lifts to a loop $l$ in $\widetilde{Z}$. 
	This lift lies in the $\frac{1}{4}$-neighbourhood of a lift $w$ of $v$. 

	There is a retraction $\widetilde{Z}\to \lk(w)$.
	There is a further retraction $\lk(w)\to l$. 
	Hence we can see that $l$ is not null-homotopic in $\widetilde{Z}\smallsetminus\{w\}$. 
	Since $q_*$ is an injection and $q_*(l) = \gamma_v$, we get that $\gamma_v$ is not null-homotopic in $Z\smallsetminus\{v\}$. 
\end{proof}

\section{Branched covers and uncountably many groups}

Throughout this section let $p$ be a prime number. Denote by $\Gamma(\lambda)$ the full subgraph of $\Gamma$ spanned by $\lambda$.

We will now take branched covers of $Z$ to obtain the uncountable family of groups. 
To do the branching, we mimic section 21 of \cite{leary}.
There, a branched cover is obtained for any regular cover of $L$. 
In our case, we can follow a similar procedure. However, we are unable to appeal to the embedding results used in \cite{leary}. 

We will focus on the vertices which map to $(a^+, b^+)$. 
At these vertices, the ascending link is given by $\Gamma(A^-\sqcup B^-)$ and the descending link is given by $\Gamma(A^+\sqcup B^+)$. 
Let $\Lambda$ be the subgraph of $\Gamma$ containing both of these graphs and one edge joining them. 
There is a retraction $\Gamma\to \Lambda$ and a further retraction $\Lambda\to S^1$ mapping both loops homeomorphically to $S^1$. 

We can now pull back the $p$-fold cover of $S^1$ to $\Lambda$. 
Let $\bar{\Gamma}$ be the corresponding cover of $\Gamma$ and $p$ be the covering map.
Thus $\bar{\Gamma}$ is a normal covering space with group of deck transformations $\Z/p\Z$. 
By construction, the preimages of $\Gamma(A^+\sqcup B^+)$ and $\Gamma(A^-\sqcup B^-)$ are connected graphs isomorphic to $S^1$. 
In fact, these are $p$-fold covers of $\Gamma(A^+\sqcup B^+)$ and $\Gamma(A^-\sqcup B^-)$, respectively. 

Recall that we have a Morse function $h\colon Z\to\R$. 
Let $V$ be the set of vertices in $Z$ that map to $(a^+, b^+)$ in $X$ and $h(v)>0$. 
The set $V$ is countably infinite and there is a natural bijection with $\N$. 
Note that we are aiming to construct groups which are not virtually torsion-free.
Hence guided by \cref{conj}, we consider only positive branching points with respect to the Morse function to eliminate any possibility of $\mathbb{Z}$-periodicity in our set $W$.
Thus the set $V$ considered here is only a half-ray of the more general case, where it would naturally biject with $\mathbb{Z}$.
As a consequence, this construction is also slightly simpler.

For each $W\subset V$, let $Z_W$ be the space obtained from $Z$ by removing all vertices $w\in W$ and adding a disk to each loop in $\lk(w)$ in the image of $\bar{\Gamma}$. 
Let $T_w$ be the complex obtained from $\lk(w)$ by attaching these disks. 
Let $H_W = \pi_1(Z_W)$. 

To each $W$ there is also an associated CAT(0) cube complex as follows. 
Let $\widetilde{Z_W}$ be the universal cover of $Z_W$.
Let $\mathcal{D}$ be the collection of disks added to $Z\smallsetminus W$ to obtain $Z_W$. 
Let $\widetilde{\mathcal{D}}$ be the preimages of the added disks in $\widetilde{Z_W}$. 
There is a covering map $\widetilde{Z_W}\smallsetminus\widetilde{\mathcal{D}}\to Z\smallsetminus W$. 
By lifting the metric and taking a completion, we obtain a CAT(0) cube complex $Y_W$. 
We can view the group $H_W$ as the group of deck transformations for the branched cover $Y_W\to Z$. 
Thus there is an action of $H_W$ on $Y_W$ which is free away from vertices. 

Let $B_W$ be the barycentric subdivision of $Y_W$. 
Let $C_W$ be the subgraph of $B_W$ spanned by the vertices which are not vertices of $Z_W$. 
The graph $C_W$ is quasi-isometric to $Z_W$. 
The action of $H_W$ on $Z_W$ preserves $C_W$ and acts freely on this subgraph. 
The following theorem now follows from Proposition 3.3 of \cite{coulonosin}.
\begin{lemma}\label{lem:uniformprop}
	The groups $H_W$ admit a uniformly proper action on a hyperbolic space. 
\end{lemma}

To finish this section, we study a presentation for $H_W$. 
In the original complex $Z$ we had one relation for each vertex $v$. 
As shown in \cref{lem:necrels}, the words $r_v$ do not represent the trivial element of $\pi_1(Z\smallsetminus \{v\})$. 
Since in the descending link we are gluing a disk to $\gamma_v^p$ if $v\in W$, we can see that a set of relators for $H_w$ is $\{ r_v\mid v\notin W\}\cup \{r_v^p\mid v\in W\}$.

The following lemma shows that $r_v$ does not represent the trivial element of $H_W$ if $V\in W$. 
This will be useful for showing that these groups are not virtually torsion-free. 
\begin{lemma}\label{lem:tors}
	If $v\in W$, then the loop $\gamma_v$ is not null-homotopic in $Z_W$. 

	Also $r_v$ does not represent the trivial word in $H_W$. 
\end{lemma}
\begin{proof}
	As before, let $\widetilde{Z_W}$ be the universal cover of $Z_W$. 
	Let $Y_W$ be the cube complex obtained by removing $\widetilde{\mathcal{D}}$ from $\widetilde{Z_W}$ and taking a completion. 
	Suppose that $\gamma_v$ is trivial in $Z_W$. 
	Then $\gamma_v$ lifts to $\widetilde{Z_W}$ in the neighbourhood of a vertex $w$. 
	There is a retraction $Y_W\smallsetminus\{v\}\to \lk(v, Y_W)$. 
	
	Let $D$ be a disk in $\widetilde{\mathcal{D}}$. 
	Since $Y_W$ is contractible, we can see that $\partial D$ is a null-homotopic loop in $Y_W$.
	Thus under the retraction, the boundary of each disk $D$ is sent to a trivial loop in $\lk(v, Y_W)$. 

	We can now obtain a retraction $\widetilde{Z_W}\to L_v$ where $L_v$ is the space obtained from $\lk(v, Y_W)$ by attaching those disks in $\widetilde{\mathcal{D}}$ which are based at $v$. 
	Since this is a retraction, it gives a surjection on fundamental groups. 
	Thus we can see that $L_v$ is simply connected. 

	We also know that $L_v$ is a cover of $T_v$, with $\pi_1(T_v)$ being $\Z/p\Z$, generated by $\gamma_v$. 
	Thus we can see that $\gamma_v$ does not lift to $L_v$ and hence does not lift to $\widetilde{Z_W}$. 
	Thus $\gamma_v$ is not null-homotopic in $Z_W$.
	Since $\gamma_v$ represents the word $r_v$ in $H_W$, we see that $r_v$ does not represent the identity in $H_W$. 
\end{proof}

Using the invariant defined in \cite{kropholler}, we can see that there are uncountably many isomorphism classes of groups $H_W$. 
Recall the following definition. 
\begin{definition}
	Let $R\subset F(A)$ be a set of words in the free group on a finite set $A$. 
	Let $G$ be the group generated by $A$. 
	Define $\mathcal{R}(G, A, R) = \{w\in R\mid r(A) = 1$ in $G\}$. 
\end{definition}

In \cite{kropholler}, it is shown that for a fixed set $R$ and fixed group $G$, there are only countably many possibilities for $\mathcal{R}(G, A, R)$. 
Let $A$ be the generating set for $H_W$ coming from $\pi_1(Z_0)$. 
Let $R = \{r_v\mid v\in Z\}$. 
From \cref{lem:tors}, we can see that $\mathcal{R}(H_W, A, R) = \{r_v\mid v\notin W\}$. 
Thus we obtain the following, 

\begin{proposition}
	The groups $H_W$ form an uncountable family of groups. 
\end{proposition}

This allows us to obtain the first part of \ref{mainthm}. 
Namely, 
\begin{theorem}
	There are groups of the form $H_W$ which cannot be subgroups of hyperbolic groups.  

	As such, there are finitely generated groups acting uniformly properly on hyperbolic spaces which cannot be subgroups of hyperbolic groups. 
\end{theorem}
\begin{proof}
	Since hyperbolic groups are finitely presented, there are only countably many of them. 
	In a given hyperbolic group $G$ there are only countably many finitely generated subgroups. 
	Thus there are only countably many finitely generated subgroups of hyperbolic groups. 
	Since there are uncountably many groups of the form $H_W$, we obtain examples of finitely generated groups which act uniformly properly on hyperbolic spaces which are not subgroups of hyperbolic groups. 
\end{proof}

\section{Virtually torsion-free criterion}

We will now prove that only countably many of the groups $H_W$ are virtually torsion-free. Denote by $o(g)$ the order of a group element.

Let $r_v$ be the sequence of words coming from $\gamma_v$ as $v$ runs over the vertices of $Z$. 
We know that there is a generating set $A$ such that $H_W$ has the presentation $\langle A\mid R\rangle$, where $R = \{r_v\mid v\notin W\}\cup \{r_v^p \mid v\in W\}$. 
By \cref{lem:tors} we know that $r_v$ is not trivial if $v\in W$. 
Thus it has order $p$ in this case. 

Let $G$ be a group and $\phi\colon H_W\to G$ be a homomorphism. 
If $v\notin W$, then $o(\phi(r_v)) = 1$. 
If $v\in W$, then $o(\phi(r_v))\ |\ p$ and so is either 1 or $p$. 
Let $O$ be the subset of $V$ consisting of those $r_v$ such that $o(\phi(r_v)) = p$. 

Now suppose that $H_W$ is virtually torsion-free. 
Thus $H_W$ contains a finite index subgroup $F_W$ which is torsion-free. 
By considering the action of $H_W$ on the cosets of $F_W$, we obtain a homomorphism $\phi\colon H_W\to S_n$, where $n = |H_W:F_W|$.

Since $F_W$ is torsion-free, we can see that if $v\in W$, then $r_v\notin F_W$.
Hence, we obtain $o(\psi(r_v)) = p$ for all $v\in W$. 
Thus $O = W$ in this case. 

The homomorphism $\phi$ is determined by a map $A\to S_n$. 
There are only finitely many such maps for a fixed $n$. 
And thus only countably many such maps as $n$ varies. 

The set $O$ is determined by the map $A\to S_n$. 
Thus there can only be countably many sets $O$ picked out by this process.
There are uncountably many groups $H_W$, with one for each $W\subset V$. 
Thus only countably many of them can be torsion-free. 

This completes the proof of \cref{mainthm}.

\bibliography{bib}{}
\bibliographystyle{plain}

\end{document}